\documentclass[10pt]{amsart}
\usepackage{amsfonts,amssymb,amscd,amsmath,enumerate,verbatim,calc,amsthm}
\input xy
\xyoption{all}

\headsep=1cm \numberwithin{equation}{section}
\newtheorem{thm}{Theorem}[section]
\newtheorem{cor}[thm]{Corollary}
\newtheorem{lem}[thm]{Lemma}
\newtheorem{prop}[thm]{Proposition}
\newtheorem{defn}[thm]{Definition}

\newtheorem{exam}[thm]{Example}
\newtheorem{rem}[thm]{Remark}

\newtheorem{nota}[thm]{Notation}

\newcommand{\Ann}{\mbox{Ann}\,}
\newcommand{\width}{\mbox{width}\,}
\newcommand{\vdim}{\mbox{vdim}\,}
\newcommand{\coker}{\mbox{Coker}\,}
\newcommand{\Hom}{\mbox{Hom}\,}
\newcommand{\Ext}{\mbox{Ext}\,}
\newcommand{\Tor}{\mbox{Tor}\,}
\newcommand{\Spec}{\mbox{Spec}\,}

\newcommand{\Ass}{\mbox{Ass}\,}

\newcommand{\Att}{\mbox{Att}\,}
\newcommand{\Supp}{\mbox{Supp}\,}

\newcommand{\depth}{\mbox{depth}\,}
\renewcommand{\dim}{\mbox{dim}\,}

\newcommand{\grd}{\mbox{grade}\,}

\newcommand{\pd}{\mbox{pd}\,}
\newcommand{\id}{\mbox{id}\,}
\newcommand{\fd}{\mbox{fd}\,}
\newcommand{\gd}{\mbox{G--dim}\,}
\newcommand{\mg}{\mbox{mag}\,}
\newcommand{\cass}{\mbox{Coass}\,}

\newcommand{\h}{\mbox{ht}\,}

\newcommand{\E}{\mbox{E}}

\renewcommand{\H}{\mbox{H}}

\newcommand{\Soc}{\mbox{Soc}\,}

\newcommand{\fa}{\mathfrak{a}}

\newcommand{\fm}{\mathfrak{m}}
\newcommand{\fp}{\mathfrak{p}}

\newcommand{\fn}{\mathfrak{n}}

\begin{document}
\bibliographystyle{amsplain}


\title[Codualizing modules]
 {Codualizing modules}

\bibliographystyle{amsplain}

     \author[M. Rahmani]{Mohammad Rahmani}
     \author[A.- J. Taherizadeh]{Abdoljavad Taherizadeh}

\address{Faculty of Mathematical Sciences and Computer,
Kharazmi University, Tehran, Iran.}

\email{m.rahmani.math@gmail.com}
\email{taheri@khu.ac.ir}



\begin{abstract}
Let $(R, \fm)$ be a Noetherian local ring. In this paper, we introduce a dual notion for dualizing modules, namely codualizing modules. We study the basic properties of codualizing modules and use them to establish an equivalence between the category of noetherian modules of finite projective dimension and the category of artinian modules of finite projective dimension. Next, we give some applications of codualizing modules. Finally, we present a mixed identity involving quasidualizing module that characterize the codualizing module. As an application, we obtain a necessary and sufficient condition for $R$ to be Gorenstein. 
\end{abstract}

\maketitle

\bibliographystyle{amsplain}
\section{introduction}

Throughout this paper, $(R, \fm , k)$ is a commutative Noetherian local ring with non-zero identity. A finitely generated $R$-module $C$ is semidualizing if the natural homothety map $ R\longrightarrow \Hom_R(C,C) $ is an isomorphism and $ \Ext^i_R(C,C)=0 $ for all $ i>0 $. Semidualizing modules have been studied by Foxby \cite{F}, Vasconcelos \cite{V}
and Golod \cite{G}. A semidualizing $R$-module $C$ is called dualizing if $ \id_R(C) < \infty $. Recently, B. Kubik \cite{K}, introduced the dual notion of semidualizing modules, namely quasidualizing modules. An Artinian $R$-module $T$ is quasidualizing if the natural homothety map $ \widehat{R} \longrightarrow \Hom_R(T,T) $ is an isomorphism and $ \Ext^i_R(T,T)=0 $ for all $ i>0 $. We denote the class of semidualizing and quasidualizing $R$-modules by $ \mathfrak{S}_0(R) $ and $ \mathfrak{Q}_0(R) $, respectively. In \cite[Theorem 3.1]{K}, Kubik showed that if $R$ is complete, then the Matlis duality functor, that is $ \Hom_R(-,E(k)) $, provides an inverse bijection between $ \mathfrak{S}_0(R) $ and $ \mathfrak{Q}_0(R) $. In Theorem 3.1, we replace $ E(k) $ with a suitable quasidualizing $R$-module. More precisely, we prove the following:\\
\textbf{Theorem 1.} Let $C$ be a semidualizing and $T , T'$ are quasidualizing $R$-modules.
\begin{itemize}
	\item[(i)]{If $ T \in \mathcal{B}_C $, then $ \Hom_R(C,T) \in \mathfrak{Q}_0(R)$.}
	\item[(ii)]{If $R$ is complete and $ T' \in \mathcal{B}_T $, then $ \Hom_R(T',T) \in \mathfrak{S}_0(R)$.}
\end{itemize}
Next, In section 4, we define a dual notion for dualizing modules, namely codualizing modules. A quasidualizing $R$-module $T$ is said to be codualizing if
$ \pd_R(T) < \infty $. We state and prove some basic facts about codualizing modules. In this section, we discuss about the existence and uniqueness of codualizing modules over Cohen-Macaulay rings that are homomorphic image of Gorenstein rings. Also, as a dual of a theorem of R.Y. Sharp \cite[Theorem 2.9]{S}, we obtain an inverse equivalence between the category of noetherian modules of finite projective dimension $ \mathcal{P}^{noet}(R) $ and the category of artinian modules of finite projective dimension  $ \mathcal{P}^{art}(R) $. The following is Theorem 4.14. \\
\textbf{Theorem 2.} Let $R$ be complete and let $T$ be codualizing. There is an inverse equivalence of categories $ \Hom_R(-,T) :  \mathcal{P}^{art}(R) \longrightarrow \mathcal{P}^{noet}(R)$ with the functor being its own quasi-inverse. 

In the last section, we obtain a characterization of codualizing modules in terms of isomorphisms involving $ \Ext $ and $ \Tor $. More precisely, the following is Theorem 5.12. \\
\textbf{Theorem 3.} Let $T$ ba a quasidualizing $R$-module and let $\dim(R)=d$. The following are equivalent:
\begin{itemize}
	\item[(i)]{$T$ is codualizing.}
	\item[(ii)]{One has
		$$\Tor^R_i(T,\Ext^j_R(T,k)) \cong \left\lbrace
		\begin{array}{c l}
		k\ \ & \text{ \ \ \ \ \ $i=j=d$,}\\
		0\ \   & \text{  \ \ $\text{ \ \ otherwise}$.}
		\end{array}
		\right.$$
	}
	\item[(iii)]{One has
		$$\Ext_R^i(T,\Tor^R_i(T,k)) \cong \left\lbrace
		\begin{array}{c l}
		k\ \ & \text{ \ \ \ \ \ $i=j=d$,}\\
		0\ \   & \text{  \ \ $\text{ \ \ otherwise}$.}
		\end{array}
		\right.$$
	}
\end{itemize}
We use this result to characterize Gorenstein rings. 
\section{preliminaries}

In this section, we recall some definitions and facts which are needed throughout this
paper. For an $R$-module $M$, the injective hull of $M$, is always denoted by $E(M)$. Also we use $ (-)^{\vee} = \Hom_R(-,E(R/\fm)) $ for the Matlis duality functor and use $ \nu_R(M) $ for the cardinal number of minimal generating set of a finitely generated $R$-module $M$, that is, 
$ \nu_R(M) = \vdim_k(k \otimes_R M) $. The depth of a
(not necessarily finitely generated) $ R $-module $ M $ is \\
\centerline{$ \depth_R(M) = \inf\{ i \geq 0 \mid \Ext_R^i(R/\fm,M) \neq 0 \} $.}
Also the width of $M$ is defined to be \\
\centerline{$ \width_R(M) = \inf\{ i \geq 0 \mid \Tor^R_i(R/\fm,M) \neq 0 \} $.} 
If $M$ is finitely generated with $ \depth_R(M) = t $, then the type of $M$, denoted by $ r_R(M) $, is defined to be
$  r_R(M) = \vdim_k(\Ext^t_R(k,M)) $. If $ \dim_R(M) = \depth_R(M) $, then $M$ is said to be a Cohen-Macaulay (abbreviated to CM) $R$-module. If, in addition, $ \dim_R(M) = \dim(R) $, then $M$ is said to be a maximal Cohen-Macaulay (abbreviated to MCM) $R$-module.
\begin{defn}
 \emph{A finitely generated $ R $-module $ C $ is \textit{semidualizing} if it satisfies the following conditions:
\begin{itemize}
             \item[(i)]{The natural homothety map $ \chi_R^C : R \longrightarrow \Hom_R(C,C) $ is an isomorphism.}
             \item[(ii)]{$ \Ext^i_R(C,C)=0 $ for all $ i>0 $.}
          \end{itemize}
         For example $R$ itself is semidualizing. An $R$-module $D$ is \textit{dualizing} if it is semidualizing and that $\id_R (D) < \infty $. For example, the canonical module of a Cohen-Macaulay local ring, if exists, is dualizing.}

\emph{Following \cite{K}, an Artinian $ R $-module $ T $ is called \textit{quasidualizing} if it satisfies the following conditions:
\begin{itemize}
	\item[(i)]{The natural homothety map $ \chi_{\widehat{R}}^T : \widehat{R} \longrightarrow \Hom_R(T,T) $ is an isomorphism.}
	\item[(ii)]{$ \Ext^i_R(T,T)=0 $ for all $ i>0 $.}
\end{itemize}}
\end{defn}
For example $ E(R/\fm) $ is a quasidualizing $ R $-module. Also, if $R$ is Artinian, then it is quasidualizing.

\begin{defn}
	\emph{Let $X$ be an $R$-module. The \textit{Auslander class with
			respect to} $X$ is the class $\mathcal{A}_X$ of $R$-modules $M$ such that:
		\begin{itemize}
			\item[(i)]{$\Tor_i^R(X,M) = 0 = \Ext^i_R(X, X \otimes_R M)$ for all $i \geq 1$, and}
			\item[(ii)]{The natural map $ \omega_{XXM} : M \rightarrow \Hom_R(X , X \otimes_R M )$ is an isomorphism.}
		\end{itemize}		
		The \textit{Bass class with
			respect to} $X$ is the class $\mathcal{B}_X$ of $R$-modules $M$ such that:
		\begin{itemize}
			\item[(i)]{$\Ext^i_R(X,M) = 0 = \Tor_i^R(X, \Hom_R(X,M))$ for all $i \geq 1$, and}
			\item[(ii)]{The natural map $ \theta_{XXM} : X \otimes_R \Hom_R(X,M) \rightarrow M $ is an isomorphism.}
		\end{itemize}
		Following \cite{K}, the class of \textit{derived $X$-reflexive} $R$-modules, denoted by $ \mathcal{G}_X $, consists of those $R$-modules $M$ for which
			\begin{itemize}
				\item[(i)]{$\Ext^i_R(M,X) = 0 = \Ext^i_R(\Hom_R(M,X), X)$ for all $i \geq 1$, and}
				\item[(ii)]{The natural biduality map $ \delta_{MXX} : M \rightarrow \Hom_R(\Hom_R(M,X),X) $ is an isomorphism.}
			\end{itemize}
					If $C$ is semidualizing, then $\mathcal{A}_C$ contains all $R$-modules of finite projective dimension and  $\mathcal{B}_C$ contains all $R$-modules of finite injective dimension. Also, if any two $ R $-modules in a short exact sequence are in $ \mathcal{A}_C $ (resp. $ \mathcal{B}_C $), then so is the third (see \cite[Corollary 6.3]{HW}). If 
					$X=R$, then the class $ \mathcal{G}_R $ is the so-called totally reflexive $R$-modules. Also, if $T$ is quasidualizing, then the class $ \mathcal{G}_T $ is introduced in \cite{K} as a generalization of Matlis reflexive modules.}
\end{defn}
Recall that an element $ x \in R $ is said to be coregular on the $R$-module $ M $ if the map $ M \overset{x}\longrightarrow M$ is surjective. Also, the sequence $ x_1, \cdots ,x_n $ is said to be a $M$-coregular sequence if $ x_i $ is $ 0 \underset{M}: (x_1,\cdots,x_{i-1}) $-coregular for all $ 1\leq i \leq n $. 
\begin{lem}
	Let $x \in \fm$ be $R$-regular and $T \in \mathfrak{Q}_0(R)$. Then for all $i > 0$, we have\\
	\centerline{$ \emph{\Ext}_R^i(R/xR,T) = 0 = \emph{\Ext}_R^i(\emph{\Hom}_R(R/xR,T),T) $}
\end{lem}
\begin{proof}
We can assume that $R$ is complete. First, since $ \fd_R(R/xR) = 1 < \infty $, and that $ T^{\vee} \in \mathfrak{S}_0(R) $ by \cite[Theorem 3.1]{K}, we have $ \Ext_R^i(R/xR,T)^{\vee} \cong \Tor_i^R(R/xR,T^{\vee}) = 0 $ for all $i > 0$ by \cite[Theorem 3.2.13]{EJ1}. Hence $ \Ext_R^i(R/xR,T) = 0 $ for all $i > 0$. Next, note that by \cite[Theorem 10.66]{Ro}, there is a third quadrant spectral sequence \\
\centerline{$ \E^{p,q}_2 = \Ext^p_R\big(\Ext_R^q(R/xR,T) , T\big) \underset{p}\Longrightarrow \Tor_n^R\big(R/xR,\Hom_R(T,T)\big) $.}
But $ \Ext_R^q(R/xR,T) = 0 $ for all $ q > 0 $ and hence $ \E^{p,q}_2 $ collapses on the $ p $-axis. So that, for all $ i > 0 $, we get \\
\centerline{$ \Ext_R^i(\Hom_R(R/xR,T),T) \cong  \Tor_i^R(R/xR,R) = 0$.}
\end{proof}
\begin{prop}\label{B0}
Let $C \in \mathfrak{S}_0(R)$ and $T \in \mathfrak{Q}_0(R)$. Then we have the following:
\begin{itemize}
           \item[(i)]{$\emph\dim(C) = \emph\dim(R)$ and $\emph{\Ass}_R(C) = \emph{\Ass}(R)$. If, in addition, $R$ is complete, then
           	$  \emph{\Ass}(C) = \emph{\Att}_R(T) $.}
           \item[(ii)]{For a non-zero $R$-module $M$, we have $ C \otimes_R M \neq 0 $ and that 
           	$ \emph{\Hom}_R(C,M) \neq 0 \neq \emph{\Hom}_R(M,T) $.}
           \item[(iii)] {An element $x \in \fm$ is $R$-regular $(C\emph{-}$regular$)$ if and only if it is $T$-coregular.}
            \item[(iv)]{ Suppose that $ x \in \fm $ is $R$-regular. Then $C/ xC \in \mathfrak{S}_0(R/xR)$ and $\emph{\Hom}_R(R/xR,T) \in \mathfrak{Q}_0(R/xR)$.}
             \item[(v)]{$\emph{\depth}_R (C) = \emph{\depth}(R) = \emph{\width}_R(T) $.}
             \item[(vi)]{$ C $ and $T$ are both indecomposable.}
             \end{itemize}
\end{prop}
\begin{proof}
 In part (i), the first two equalities follow from the definition of $ C $. For the last equality, use \cite[Theorem 3.1]{K} and \cite[Corollary 10.2.20]{BS}. For (ii), $ C \otimes_R M \neq 0 $ and $\Hom_R(C,M) \neq 0$ both follow from the fact that 
 $ \Supp_R(C) = \Spec(R) $. Also $ \Hom_R(M,T) \neq 0 $ follows from \cite[Lemma 3.11]{K}. For (iii), we can assume that $R$ is complete. Now, note that for an element $x \in \fm$, 
 $ R \overset{x}\longrightarrow R $ is a monomorphism if and only if $ T \overset{x}\longrightarrow T $ is an epimorphism since by Lemma 2.3, 
 $ \Ext_R^1(R/xR,T) = 0 $. In (iv), the fact that $ C/ xC \in \mathfrak{S}_0(R/xR) $ is well-known. For the other claim, note that the exact sequence 
 \centerline{$ 0 \longrightarrow R \overset{x}\longrightarrow R \longrightarrow R/xR \longrightarrow 0 $,}
 in conjunction with the five lemma and Lemma 2.3 show that $ \theta_{R/xRTT} $ is an isomorphism.
  Consider the following commutative diagram
\begin{displaymath}
 \xymatrix{
 	\widehat{R}/x\widehat{R}  \ar@{->}[rr]^-{\chi_{\widehat{R}}^T \otimes_R R/xR} \ar@{=}[dd] &&  R/xR \otimes_R \Hom_R(T,T) \ar[d]^{\cong}_{\theta_{R/xRTT}} \\
 	&&  \Hom_R\big(\Hom_R(R/xR,T) , T \big) \ar[d]^{\cong} \\
 	\widehat{R}/x\widehat{R} \ar@{->}[rr]_-{\chi_{\widehat{R}/x\widehat{R}}^{\emph{Hom}(R/xR,T)}} &&  \Hom_{R/xR}\big(\Hom_R(R/xR,T) , \Hom_R(R/xR,T) \big)}
\end{displaymath}
in which the unspecified isomorphism is from Hom-tensor adjointness. Hence, $ \chi_{\widehat{R}/x\widehat{R}}^{Hom(R/xR,T)} $ is an isomorphism. Next, by \cite[Theorem 10.64]{Ro}, there exists a third quadrant spectral sequence\\
\centerline{$ \E^{p,q}_2 = \Ext^p_{R/xR}\big(\Hom_R(R/xR,T) , \Ext^q_R(R/xR,T)\big) \underset{p}\Longrightarrow \Ext^n_R\big(\Hom_R(R/xR,T), T\big) $.}
On the other hand, by Lemma 2.3, we have $ \Ext^q_R(R/xR,T) = 0 $ for all $ q > 0 $. Therefore, 
$ \E^{p,q}_2 $ collapses on the 
$ p $-axis and we have the isomorphism\\
\centerline{$  \Ext^n_{R/xR}\big(\Hom_R(R/xR,T) , \Hom_R(R/xR,T)\big) \cong \Ext^n_R\big(\Hom_R(R/xR,T), T\big)  = 0 $,}
for all $ n > 0 $, where the last equality holds by lemma 2.3.

(v). The left hand side equality is well-known. For the other equality, we can assume that $R$ is complete and then proceed by induction on $d = \depth(R)$. If $d = 0$, then 
$ \Hom_{\widehat{R}}(k,{\widehat{R}}) \neq 0 $, and so $ \Hom_R(k \otimes_{\widehat{R}} T,T) \cong \Hom_{\widehat{R}}(k,\Hom_R(T,T)) \neq 0 $. In particular, $ k \otimes_R T \neq 0 $, whence $ \width_R(T) = 0 $. Now, assume inductively that $d > 0$. Choose $ x \in \fm $ to be $R$-regular. By (iii), $x$ is $T$-coregular, and hence there exists an exact sequence\\
\centerline{$ 0 \longrightarrow \Hom_R(R/xR,T) \longrightarrow T \overset{x}\longrightarrow T \longrightarrow 0 $.}
By (iv), $ \Hom_R(R/xR,T) $ is a quasidualizing $R/xR$-module, and so by induction hypothesis $ \depth(R/xR) = \width_{R/xR}(\Hom_R(R/xR,T))$. Finally, if $ t = \width_R(T) $, the long exact sequence\\
\centerline{$ \cdots \rightarrow \Tor^R_{i+1}(k,T) \overset{x}\rightarrow \Tor^R_{i+1}(k,T) \rightarrow \Tor^R_i(k,\Hom_R(R/xR,T)) \rightarrow \Tor^R_i(k,T) \overset{x}\rightarrow \cdots $,}
shows that $ \Tor^R_i(k,\Hom_R(R/xR,T)) = 0 $ for all $ i < t-1 $ and since $ \Tor^R_t(k,T) \neq 0 $ we have $ \Tor^R_{t-1}(k,\Hom_R(R/xR,T)) \neq 0$. 
 Consequently, \\
 \centerline{$ \width_{R/xR}(\Hom_R(R/xR,T)) = \width_R(\Hom_R(R/xR,T)) = t - 1 $,}
  and so $ d = t $, as wanted.

(vi). Note that if $C$ (resp. $T$) is decomposable, then $R$ (resp. $\widehat{R}$) must be decomposable which is impossible since $R$ (resp. $\widehat{R}$) is local.
\end{proof}
\begin{lem}
Let $R$ be complete and $M$ be a Matlis reflexive $R$-module of finite projective dimension. Then $ M \in \mathcal{G}_T $.
\end{lem}
\begin{proof}
See \cite[Proposition 3.9]{K}.
\end{proof}
\section{Quasidualizing modules via semidualizing modules}
In this section, our aim is to generalize \cite[Theorem 3.1]{K}, which says that if $R$ is complete, then the Matlis dual functor provides an inverse bijection between the classes  $\mathfrak{S}_0(R)$ and $\mathfrak{Q}_0(R)$.
\begin{thm}
	Let $C$ be a semidualizing and $T , T'$ are quasidualizing $R$-modules.
	\begin{itemize}
		\item[(i)]{If $ T \in \mathcal{B}_C $, then $ \emph{\Hom}_R(C,T) \in \mathfrak{Q}_0(R)$.}
		\item[(ii)]{If $R$ is complete and $ T' \in \mathcal{B}_T $, then $ \emph{\Hom}_R(T',T) \in \mathfrak{S}_0(R)$.}
		\item[(iii)] {If $ T \in \mathcal{A}_C $, then $C \otimes_R T \in \mathfrak{Q}_0(R)$.}
	\end{itemize}
\end{thm}
\begin{proof}
(i)
By assumption, $ \theta_{CT} : C \otimes_R \Hom_R(C,T) \rightarrow T$ is an isomorphism. Now consider the following commutative diagram
\begin{displaymath}
\xymatrix{
	\widehat{R}  \ar@{->}[rr]^-{\chi_{\widehat{R}}^T} \ar@{=}[dd] &&  \Hom_R(T,T) \ar[d]^{\cong}_{\emph{Hom}(\theta_{CT},T)} \\
	&& \Hom_R\big(C \otimes_R \Hom_R(C,T) ,T\big) \ar[d]^{\cong} \\
	\widehat{R} \ar@{->}[rr]_-{\chi_{\widehat{R}}^{\emph{Hom}(C,T)}} &&  \Hom_R\big(\Hom_R(C,T) ,\Hom_R(C,T)\big)}
\end{displaymath}
in which the unlabeled isomorphism is Hom-tensor adjointness. Therefore, $ \chi_{\widehat{R}}^{\emph{Hom}(C,T)} $ must be an isomorphism. Next, consider the following third quadrant spectral sequence \cite[Theorem 10.62]{Ro},\\
\centerline{$ \E^{p,q}_2 = \Ext^p_R(\Tor^R_q(C,\Hom_R(C,T)) , T) \underset{p}\Longrightarrow \Ext^n_R(\Hom_R(C,T),\Hom_R(C,T)) $.}
But, since $ T \in \mathcal{B}_C $, we have $ \Tor^R_q(C,\Hom_R(C,T)) = 0 $ for all $ q > 0 $. So that $ \E^{p,q}_2 $ collapses on the 
$ p $-axis and we have the isomorphisms\\
\centerline{$  \Ext^n_R(\Hom_R(C,T),\Hom_R(C,T)) \cong \Ext^n_R(C \otimes_R \Hom_R(C,T),T) \cong \Ext^n_R(T,T) = 0 $,}
for all $ n \geq 0 $.

(ii). First, note that $ \Hom_R(T',T) $ is a finitely generated $R$-module by \cite[Lemma 2.1]{KLW1}. Now, the proof is similar to the part (i).

(iii) By assumption,  $ \omega_{CT} :  T \rightarrow \Hom_R(C,C \otimes_R T) $ is an isomorphism. Now consider the following commutative diagram
\begin{displaymath}
\xymatrix{
	\widehat{R}  \ar@{->}[rr]^-{\chi_{\widehat{R}}^T} \ar@{=}[dd] &&  \Hom_R(T,T) \ar[d]^{\cong}_{\emph{Hom}(T,\omega_{CT})} \\
	&& \Hom_R\big(T ,\Hom_R(C,C \otimes_R T) \big) \ar[d]^{\cong} \\
	\widehat{R} \ar@{->}[rr]_-{\chi_{\widehat{R}}^{C \otimes T}} &&  \Hom_R\big(C \otimes_R T , C \otimes_R T \big)}
	\end{displaymath}
in which the unlabeled isomorphism is Hom-tensor adjointness. Therefore, $ \chi_{\widehat{R}}^{C \otimes T} $ must be an isomorphism.
Next, consider the following third quadrant spectral sequence \cite[Theorem 10.64]{Ro},\\
\centerline{$ \E^{p,q}_2 = \Ext^p_R(T , \Ext^q_R(C , C\otimes_R T)) \underset{p}\Longrightarrow \Ext^n_R(C \otimes_R T,C \otimes_R T) $.}
But, since $ T \in \mathcal{A}_C $, we have $ \Ext^q_R(C , C\otimes_R T) = 0 $ for all $ q > 0 $. So that $ \E^{p,q}_2 $ collapses on the 
$ p $-axis and we have the isomorphisms\\
\centerline{$  \Ext^n_R(C \otimes_R T,C \otimes_R T) \cong \Ext^n_R(T,\Hom(C , C\otimes_R T)) \cong \Ext^n_R(T,T) = 0 $,}
for all $ n \geq 0 $.
\end{proof}
\begin{cor}
$($See \cite[Theorem 3.1]{K}$)$	Let $R$ be complete. Then $ (-)^{\vee} : \mathfrak{S}_0(R) \longrightarrow \mathfrak{Q}_0(R) $ and
$ (-)^{\vee} : \mathfrak{Q}_0(R) \longrightarrow \mathfrak{S}_0(R) $  are inverse bijections.
\end{cor}
\begin{proof}
In the Theorem 3.1,	set $ T = T' = E(R/\fm) $.
\end{proof}
\begin{thm}
	Let $T$ be a quasidualizing $R$-module.
		\begin{itemize}
			\item[(i)]{$ T \otimes_R T \neq 0 $ if and only if $ \emph{\dim}(R) = 0 $.}
			\item[(ii)]{$ T \otimes_R T \in \mathfrak{Q}_0(R)$ if and only if $ T \cong R $.}
		\end{itemize}
\end{thm}
\begin{proof}
(i). Assume that $ T \otimes_R T \neq 0 $. By \cite[Corollary 7.4]{KLW}, $ \ell_R(T \otimes_R T) < \infty $ and then, by \cite[Theorem 3.1]{K}, $ (T \otimes_R T)^{\vee} $ is a semidualizing $\widehat{R}$-module of finite length. It follows, by Proposition 2.4(i), that
$ \dim(R) = \dim(\widehat{R}) = \dim_{\widehat{R}}((T \otimes_R T)^{\vee}) = 0 $, as wanted. Conversely, assume that $\dim(R) = 0$. Then, $R$ is complete with respect to $\fm$-adic topology and that $T$ is noetherian. Thus, in particular, $T$ is semidualizing and so by Proposition 2.4(ii), we have
 	$ T \otimes_R T \neq 0 $.
 	
(ii). Assume that $ T \otimes_R T $ is quasidualizing. Then, by (i), $\dim(R) = 0$. Then, $R$ is complete with respect to $\fm$-adic topology. Also, by \cite[Theorem 3.1]{K}, $ T \otimes_R T $ is semidualizing. Hence, by \cite[Theorem 3.2]{FW}, $T \cong R$. The converse is evident.
\end{proof}
\begin{cor}
$ ( $See \cite{EJ2}$ ) $	Assume that $\fp$ is a prime ideal of a $ ( $not necessarily local$ ) $ ring $R$. 
			\begin{itemize}
				\item[(i)]{$ E(R/\fp) \otimes_R E(R/\fp) \neq 0 $ if and only if $ \emph{\dim}(R_{\fp}) = 0 $.}
				\item[(ii)]{$ E(R/\fp) \otimes_R E(R/\fp) \cong E(R/\fp)$ if and only if $R_{\fp}$ is Artinian Gorenstein.}
			\end{itemize}
\end{cor}
\begin{proof}
	(i). Note that $ E(R/\fp) $ is a quasidualizing $ R_{\fp} $-module. Now the result follows from Theorem 3.3(i)
	
	(ii). If  $ R_{\fp} $ is Artinian Gorenstein, then $ R_{\fp} \cong E(R/\fp) $ and so $ E(R/\fp) \otimes_R E(R/\fp) \cong E(R/\fp) $ is quasidualizing. Conversely, if $ E(R/\fp) \otimes_R E(R/\fp) \cong E(R/\fp)$, then since $ E(R/\fp) $ is quasidualizing, by Theorem 3.3(ii), we have $ R_{\fp} \cong E(R/\fp) $, whence $R_{\fp}$ is Artinian Gorenstein.
\end{proof}

\section{Existence and uniqueness}
\begin{prop}
	Let $T$ be a quasidualizing $R$-module.
	\begin{itemize}
		\item[(i)]{If $\emph{\id}_R(T) < \infty$, then $T \cong E(R/\fm)$.}
		\item[(ii)] {If $\emph{\pd}_R(T) < \infty$, then $\emph{\pd}_R(T) = \emph{\dim}(R)$.}
	\end{itemize}
\end{prop}
\begin{proof}
	(i).
	Note that $T^{\vee}$ is a semidualizing $ \widehat{R} $-module by \cite[Theorem 3.1]{K}. Also,  note that 
	$ \pd_{\widehat{R}}(T^{\vee}) < \infty $. Hence, by Proposition 2.4(v) and the Auslander-Buchsbaum formula, we have 
	$ \pd_{\widehat{R}}(T^{\vee}) = 0 $, whence $T^{\vee}$ is a finitely generated free $ \widehat{R} $-module. But then
	$ T^{\vee\vee} \cong E(R/\fm)^t $ for some $t \in \mathbb{N}$. Since $T$ is a Matlis reflexive $ \widehat{R} $-module, we have
	$ T \cong E(R/\fm)^t $. Now, by Proposition 2.4(vi), we have  $T \cong E(R/\fm)$.
	
	(ii).
	In view of \cite[Theorem 3.1]{K}, $ T^{\vee} $ must be dualizing for $ \widehat{R} $. Now, there are (in)equalities\\
	\centerline{$ \dim(\widehat{R}) = \id_{\widehat{R}}(T^{\vee}) \leq \pd_R(T) \leq \dim(R), $}
	in which the last inequality follows from \cite[Corollary 3.2.7]{GR}. 
\end{proof}
\begin{rem}
	\emph{In \cite[Theorem 3.20]{RT}, the authors proved that a $d$-dimensional local ring $R$ is Gorenstein if and only if there exists a $d$-perfect quasidualizing $R$-module with 1-dimensional socle. Then, in \cite[Question 3.22]{RT}, they asked that can one omit the condition "1-dimensional socle"? The answer is negative as the following example shows. Let $R$ be a complete CM ring that is not Gorenstein. Let $ \omega_R $ be the dualizing module of $R$. Then we have $  \nu_R(\omega_R) \geq 2 $, since otherwise $ \omega_R $ must be cyclic, and hence $ R \cong R/\Ann_R(\omega_R)  \cong \omega_R $; so that $R$ is Gorenstein which is impossible. Now, by \cite[Theorem 3.19]{RT} and \cite[Theorem 3.1]{K},
		$ \omega_R^{\vee} $ is a $d$-perfect quasidualizing $R$-module and that\\
		\centerline{$ \vdim_k(\Soc(\omega_R^{\vee}))  = \vdim_k(k \otimes_R \omega_R) \geq 2 $.}
		An explicit example is $ R = k[[X,Y,Z]]/(X^2,XY,Y^2) $.}
\end{rem}
\begin{defn}
	\emph{Let $T$ be a quasidualizing $R$-module. We say that $T$ is a codualizing $R$-module if 
		$\pd_R(T) < \infty$.}
\end{defn}
\begin{exam}
	\emph{	If $R$ is Artinian, then it is codualizing. Also, if $R$ is Gorenstein, then $E(R/\fm)$ is codualizing. }
\end{exam}
\begin{rem}
	\emph{	Note that if there exists a codualizing $R$-module, then $R$ must be Cohen-Macaulay. Indeed, if $T$ is codualizing over $R$, then it is codualizing over $ \widehat{R} $. Now, by \cite[Theorem 3.1]{K}, there exists a semidualizing  $ \widehat{R} $-module $K$ such that
		$ T = \Hom_{\widehat{R}}(K,E(R/\fm)) $. Now, one has $ \id_{\widehat{R}}(K) \leq \pd_{\widehat{R}}(T) < \infty $. Thus, we have
		$ \dim(\widehat{R}) \leq \id_{\widehat{R}}(K) = \depth({\widehat{R}}) $, whence $ \widehat{R} $ and so $R$ is CM. Also, if $R$ is complete then $T$ is codualizing if and only if $T^{\vee}$ is dualizing by \cite[Theorem 3.1]{K} and the fact that 
		$ \pd_R(T) , \id_R(T^{\vee})$ are simultaneously finite. In particular, $R$ is Gorenstein if and only if $ E(R/\fm) $ is codualizing.}
\end{rem}
\begin{lem}
	Let $x \in \fm$ be $R$-regular, $C \in \mathfrak{S}_0(R)$ and $T \in \mathfrak{Q}_0(R)$.
	\begin{itemize}
		\item[(i)]{$C$ is dualizing for $R$ if and only if $C/xC$ is dualizing for $R/xR$.}
		\item[(ii)]{$T$ is codualizing for $R$ if and only if $\emph{\Hom}_R(R/xR,T)$ is codualizing for $R/xR$.}
	\end{itemize}
\end{lem}
\begin{proof}
	The part (i) is well-known. For (ii), we can assume that $R$ is complete. We have to show that $ \pd_R(T) < \infty $ if and only if $ \pd_{R/xR}(\Hom_R(R/xR,T)) < \infty  $. By Proposition 2.4(iii), there exists an exact sequence\\
	\centerline{$ 0 \longrightarrow \Hom_R(R/xR,T) \longrightarrow T \overset{x}\longrightarrow T \longrightarrow 0 $,}
	which, for any finitely generated $R$-module $N$, induces a long exact sequence \\
\centerline{$ \cdots \overset{x}\rightarrow \Ext^{i-1}_R(T,N) \rightarrow \Ext^i_R(T,\Hom_R(R/xR,N)) \rightarrow \Ext^i_R(T,N) \overset{x}\rightarrow \Ext^i_R(T,N) \rightarrow \cdots $.}
Note that all terms in the above long exact sequence are finitely generated by \cite[Theorem 2.2]{KLW1}.
Now, a simple use of Nakayama's lemma shows that $ \pd_R(T) < \infty$ if and only if $ \pd_R(\Hom_R(R/xR,T)) < \infty $. Finally, since\\
	\centerline{ $ \pd_{R/xR}(\Hom_R(R/xR,T))  = \pd_R(\Hom_R(R/xR,T)) - 1 $,}
	 we are done.
\end{proof}
\begin{lem}
Let $M$ and $N$ be two non-zero artinian $R$-modules. Let $x \in \fm$ be $R$-regular and $M$-coregular. Then a homomorphism
$ \varphi : M \longrightarrow N $ is an isomorphism if and only if the induced homomorphism\\
\centerline{$ \emph{\Hom}_R(R/xR,\varphi) : \emph{\Hom}_R(R/xR,M) \longrightarrow \emph{\Hom}_R(R/xR,N) $}
is an isomorphism. 	
\end{lem}
\begin{proof}
The necessity is trivial. Assume that $ \Hom_R(R/xR,\varphi) $ is an isomorphism. The exact sequence\\
\centerline{$ 0 \longrightarrow \ker(\varphi) \longrightarrow M \longrightarrow N$,}
induces an exact sequence \\
	\centerline{$ 0 \rightarrow \Hom_R(R/xR,\ker(\varphi)) \rightarrow \Hom_R(R/xR,M) \rightarrow \Hom_R(R/xR,N)$,}
	from which we conclude that $\big( 0 \underset{\ker(\varphi)}: x \big) = \Hom_R(R/xR,\ker(\varphi))  = 0$. In the following, we use induction on $i$ to show
	 that $\big( 0 \underset{\ker(\varphi)}: x^i \big) = 0 $ for all $ i \geq 1 $. Suppose, on the contrary, that 
	 $\big( 0 \underset{\ker(\varphi)}: x^i \big) \neq 0 $. Consider $ 0 \neq m \in \big( 0 \underset{\ker(\varphi)} : x^i \big) $. Then, we have 
	 $ x^im=0 $ and so $ x^{i-1}m \in \big( 0 \underset{\ker(\varphi)}: x \big)= 0 $ which implies that $ m \in \big( 0 \underset{\ker(\varphi)}: x^{i-1} \big) = 0 $, a contradiction.
	   Now, since $ \ker(\varphi) $ is artinian, we have $ \ker(\varphi) = \underset{i \geq 1}\bigcup \big(0 \underset{\ker(\varphi)}: x^i \big) = 0 $. Next, consider the exact sequence \\
	\centerline{$ 0  \longrightarrow M \longrightarrow N \longrightarrow \coker(\varphi) \longrightarrow 0$,}
	to induce an exact sequence \\
	\centerline{$ \Hom_R(R/xR,M) \rightarrow \Hom_R(R/xR,N) \rightarrow \Hom_R(R/xR,\coker(\varphi)) \rightarrow
		 \Ext_R^1(R/xR,M) = 0 $.}
		It follows that  $ \Hom_R(R/xR,\coker(\varphi)) = 0 $. A similar argument as above, shows that $ \coker(\varphi) = 0 $.
\end{proof}
\begin{lem}
Let $M$ and $N$ be two $R$-modules.	Let $ x $ be an $R$-regular and $N$-coregular element for which $xM=0$. Then for all $ i > 0 $ we have \\
\centerline{$ \emph{\Tor}^R_i(M,N) \cong \emph{\Tor}^{R/xR}_{i-1}(M,\emph{\Hom}_R(R/xR,N)) $.}
\end{lem}
\begin{proof}
Note that 	$ \Big(\Tor^R_i(-,N)\Big)_{i \in \mathbb{N}_0} $ is a positive strongly connected sequence of covariant functors. Now, apply the functor $ M \otimes_R - $ the exact sequence \\
\centerline{$ 0 \longrightarrow \Hom_R(R/xR,N) \longrightarrow N \overset{x}\longrightarrow N \longrightarrow 0 $,}
to obtain the following exact sequence \\
 \centerline{$ \Tor^R_1(M,N) \overset{x}\rightarrow \Tor^R_1(M,N) \rightarrow M \otimes_R \Hom_R(R/xR,N) \rightarrow M \otimes_R N $.}
But, $ \Tor^R_1(M,N) \overset{x}\rightarrow \Tor^R_1(M,N) $ is just the zero map, and that $ M \otimes_R N = 0 $. Thus, we have \\
 \centerline{$ \Tor^R_1(M,N) \cong M \otimes_R \Hom_R(R/xR,N) = M \otimes_{R/xR} \Hom_R(R/xR,N)$.}
 On the other hand, 
 if $F$ is a free $ R/xR $-module, then $ \pd_R(F) = 1 $ and hence $ \Tor^R_i(M,N)  = 0$ for all $i > 1$. According to \cite[Exercise 8.10]{Ro2}, the proof is complete.
\end{proof}
\begin{thm}
	Let $ R \rightarrow (S, \fn) $ be a local ring epimorphism of local rings with $R$ is Gorenstein, $S$ is \emph{CM}, $ \emph{\dim}(R) = n $ and 
	$ \emph{\dim}(S) = m $. Then $ \emph{\Tor}^R_{n-m}(S,E(R/\fm)) $ is codualizing for $S$.
\end{thm}
\begin{proof}
	One has $ S \cong R/I $ for some ideal $I$ of $R$ with $ \grd_R(I) = \h_R(I) = n-m $. Assume that $\underline{x}$ is a maximal $R$-sequence in $I$. Set $ \overline{R}=R/\underline{x}R $. By Lemma 4.8, we have the isomorphism $ \Tor^R_{n-m}(S,E(R/\fm)) \cong S \otimes_{\overline{R}}E_{\overline{R}}(R/\fm)$. Now, $ \overline{R} $ is Gorenstein and the induced homomorphism 
	$ \overline{R} \rightarrow S$ is again an epimorphism. Hence, we can replace $ R $ by $ \overline{R} $, and assume that $ n = m $. Note that $ S \otimes_R E(R/\fm) $ is a Artinian $S$-module. We have to show that $ S \otimes_R E(R/\fm) \in \mathfrak{Q}_0(S)$ and that
	$ \pd_S(S \otimes_R E(R/\fm)) < \infty $. Observe that the completion of $S$ in $\fn$-adic topology coincides with the completion of $S$ as an $R$-module since $ \fm S=\fn $. Also, we have the isomorphisms \\
	\centerline{$ S \otimes_R E(R/\fm) = S \otimes_R E_{\widehat{R}}(\widehat{R}/\fm \widehat{R}) \cong S \otimes_{\widehat{R}} E_{\widehat{R}}(\widehat{R}/\fm \widehat{R}) \cong \widehat{S} \otimes_{\widehat{R}} E_{\widehat{R}}(\widehat{R}/\fm \widehat{R})$.}
	Moreover, we have the commutative diagram\\
	\begin{displaymath}	
	\xymatrix{
		R  \ar@{->}[rr]^{\phi} \ar@{->}[d]  &&  S \ar[d] \\
		\widehat{R}  \ar@{->}[rr]_-{\widehat{\phi}}  &&  \widehat{S}}
	\end{displaymath}
		which allows us to make an additional assumption that $R$ and $S$ are complete in respective topologies.
		  In view of \cite[Theorem 3.1]{K}, we have to show that \\
		  \centerline{$ \Hom_R(S,R) \cong \Hom_S\big(S \otimes_R E_R(R/\fm),E_S(S/\fn)\big) \in \mathfrak{S}_0(S) $.}
		  First, note that 
		 $ \gd_R(S) < \infty $ by \cite[Theorem 1.4.9]{C}, and then we have the equality $ \gd_R(S)= \depth(R) - \depth_R(S) = 0 $, by \cite[Theorem 1.4.9]{C}. Hence, the natural biduality map $ \delta_{SR}: S \rightarrow \Hom_R(\Hom_R(S,R),R) $ is an isomorphism. Now, the commutative diagram 
		 \begin{displaymath}
		 \xymatrix{
		 	S  \ar@{->}[rr]^-{\delta_{SR}} \ar@{=}[dd] &&  \Hom_R(\Hom_R(S,R),R) \ar[d]^{\cong} \\
		 	&& \Hom_R(S\otimes_S \Hom_R(S,R),R) \ar[d]^{\cong} \\
		 	S \ar@{->}[rr]_-{\chi_{S}^{Hom(S,R)}} &&  \Hom_S(\Hom_R(S,R),\Hom_R(S,R))}
		 \end{displaymath} 
		 shows that $ \chi_{S}^{Hom(S,R)} $ is an isomorphism. Also, by \cite[Theorem 10.64]{Ro}, there is a third quadrant spectral sequence\\
		 \centerline{$ \E^{p,q}_2 = \Ext^p_S(\Hom_R(S,R) , \Ext^q_R(S , R)) \underset{p}\Longrightarrow \Ext^n_R(\Hom_R(S,R),R) $.}
		 But, since $ \gd_R(S) = 0 $, we have $ \Ext^q_R(S , R) = 0 $ for all $q > 0$, and then $ \E^{p,q}_2 $ collapses on the $p$-axis. So that, we have the isomorphism\\
		 \centerline{$ \Ext^n_S(\Hom_R(S,R) , \Hom_R(S , R)) \cong \Ext^n_R(\Hom_R(S,R),R) = 0$,}
		 for all $n > 0$. Now, we need only to show that $ \pd_S(S \otimes_R E(R/\fm)) < \infty$. To do this, first note that
		 $ \pd_R(E(R/\fm)) = n $. On the other hand, by \cite[Theorem 3.2.1]{EJ1}, we have \\
		  \centerline{$ \Tor^R_i(S,E(R/\fm)) \cong \Ext^i_R(S , R)^{\vee} = 0 $,}
		  for all $ i > 0 $. Hence, a bounded free resolution of $ E(R/\fm) $, by applying $ S \otimes_R -$, gives a bounded free resolution for $ S \otimes_R E(R/\fm) $ as an $S$-module, as wanted.
\end{proof}
\begin{nota}
	\emph{We denote the full subcategory of artinian (resp. noetherian) $R$-modules of finite projective dimension 
		by $ \mathcal{P}^{art}(R) $(resp. $ \mathcal{P}^{noet}(R) $).}
\end{nota}
\begin{lem}
Let $T$ be a quasidualizing $R$-module. If $M \in \mathcal{P}^{art}(R)$, then $ M \hookrightarrow T^m $ for some $ m \in \mathbb{N} $.	
\end{lem}
\begin{proof}
We can assume that $R$ is complete. Note that $\Hom_R(M,T)$ is a finitely generated $R$-module. Now, in view of Lemma 2.5, the exact sequence 
$ R^m \longrightarrow \Hom_R(M,T) \longrightarrow 0$, $m \in \mathbb{N}$, induces an exact sequence \\
\centerline{$ 0 \longrightarrow M \cong \Hom_R(\Hom_R(M,T),T) \longrightarrow T^m$,}
as wanted.
\end{proof}
\begin{thm}
	Let $ \emph{\dim}(R) = d $. Let $T$ be codualizing and let $M$ be an $R$-module. 
	\begin{itemize}
		\item[(i)]{Let $M \in \mathcal{P}^{art}(R) $ with $ \emph{\width}_R(M) = \emph{\dim}(R) $. Then $ M \cong T^m $ for some $ m \in \mathbb{N} $.}
		\item[(i)]{$M \in \mathcal{P}^{art}(R) $ if and only if there exists an exact sequence\\
			\centerline{$ 0 \longrightarrow M \longrightarrow T_0 \longrightarrow \cdots \longrightarrow T_n \longrightarrow 0 $,}
			for some $ n \in \mathbb{N} $, in which every $ T_i $ is a finite direct sum of copies of $T$.}
		\end{itemize}
\end{thm}
\begin{proof}
	In the rest of the proof, we can assume that $R$ is complete.
	
	(i). Choose $\underline{x}$ to be a maximal $R$-sequence in $\fm$. Then,  $\underline{x}$ is $M$-coregular as well. Now, 
	$ \Hom_R(R/\underline{x}R , M) $ is a free $ R/\underline{x}R $-module of finite length. Therefore, by Proposition 4.1(ii) and Lemma 4.6(ii), we have
	the isomorphisms
		\[\begin{array}{rl}
		\Hom_R(R/\underline{x}R , M) &\cong \Big(R/\underline{x}R\Big)^m \\
		&\cong \Hom_R(R/\underline{x}R , T)^m  \\
		&\cong \Hom_R(R/\underline{x}R , T^m),\\
		\end{array}\]
		for some $m \in \mathbb{N}$. Now, in view of Lemma 4.7, we have $ M \cong T^m $.
		
	(ii). If there exists such an exact sequence, then it easily follows that $M \in \mathcal{P}^{art}(R) $. We shall prove the converse. By  Lemma 4.12, we can write an exact sequence \\
	\centerline{$ 0 \longrightarrow M \longrightarrow T_0 \longrightarrow K_0 \longrightarrow 0 $,}
in which $T_0$ is a finite direct sum of copies of $T$. Now, if $ K_0 = 0 $, then we are done.
 If not, since $ K_0 \in \mathcal{P}^{art}(R) $, by another use of Lemma 4.12, we can write an exact sequence \\
	\centerline{$ 0 \longrightarrow K_0 \longrightarrow T_1 \longrightarrow K_1 \longrightarrow 0 $,}
	in which $T_1$ is a finite direct sum of copies of $T$. Now, we can proceed the same way to get an exact sequence \\
	\centerline{$ 0 \longrightarrow M \longrightarrow T_1 \longrightarrow \cdots \longrightarrow T_n \longrightarrow \cdots $.}
	To complete the proof, note that $ \width_R(T_d) = d $, and hence, by (i), we are done.
\end{proof}
\begin{rem}
	\emph{A codualizing module, if exists, is unique up to isomorphism. A complete CM local ring admits a codualizing module.}
\end{rem}
\begin{thm}
Let $R$ be complete with $\emph{\dim}(R) = d$, and let $T$ be codualizing. There is an equivalence of categories $ \emph{\Hom}_R(-,T) :  \mathcal{P}^{art}(R) \longrightarrow \mathcal{P}^{noet}(R)$ with the functor being its own quasi-inverse.
\end{thm}
\begin{proof}
First, note that if $M$ (resp. $A$) is noetherian (resp. artinian), then $ \Hom_R(M,T) $ (resp. $ \Hom_R(A,T) $) is artinian (resp. noetherian). Now, let $ A \in \mathcal{P}^{art}(R) $. Note that \\
\centerline{$ \Hom_R(A,T) \cong \Hom_R(A,T^{\vee \vee}) \cong  \Hom_R(T^{\vee},A^{\vee}) $.}
Now, by \cite[Theorem 10.66]{Ro}, there exists a third quadrant spectral sequence\\
\centerline{$ \E^{p,q}_2 = \Ext^p_R(\Ext^q_R(k,T^{\vee}) , A^{\vee}) \underset{p}\Longrightarrow \Tor_n^R(k,\Hom_R(T^{\vee},A^{\vee})) $.}
But, since $ T^{\vee} $ is dualizing, we have $ \Ext^q_R(k,T^{\vee}) = 0 $ for all $ q \neq d $. Therefore $ \E^{p,q}_2 $ collapses on the 
$ d $-th column, and we have the isomorphism \\
\centerline{$ \Ext^{d-n}_R(\Ext^d_R(k,T^{\vee}) , A^{\vee}) \cong \Tor_n^R(k,\Hom_R(T^{\vee},A^{\vee})) $.}
Now, since $ \Ext^d_R(k,T^{\vee}) \cong k $ and $ \id_R(A^{\vee}) = d $, we get \\
\centerline{$ \Tor_n^R(k,\Hom_R(T^{\vee},A^{\vee})) \cong \Ext^{d-n}_R(k, A^{\vee}) = 0$,}
for all $ n > d $. Consequently, $ \pd_R(\Hom_R(A,T)) = \pd_R(\Hom_R(T^{\vee},A^{\vee})) < \infty $, as wanted. Next, let  
$ M \in \mathcal{P}^{noet}(R) $. Choose a finite free resolution \\
\centerline{ $ 0 \longrightarrow F_t \longrightarrow \cdots \longrightarrow F_0 \longrightarrow M \longrightarrow 0 $ }
and apply $ \Hom_R(-,T) $ to get the complex \\
\centerline{ $ 0 \longrightarrow \Hom_R(M,T) \rightarrow \Hom_R(F_0,T) \rightarrow \cdots \rightarrow \Hom_R(F_n,T) \rightarrow 0 $, }
which is exact by Lemma 2.5. Now, $\pd_R(\Hom_R(F_i,T)) = \pd_R(T) = d$ for all $ 0 \leq i \leq n $. Thus, $ \pd_R(\Hom_R(M,T)) < \infty $, as wanted. Now, we are done by Lemma 2.5.
\end{proof}
\begin{thm}
	Let $T$ be codualizing and $M$ be an $R$-module.
		\begin{itemize}
			\item[(i)]{If $M \in \mathcal{P}^{art}(R) $, then $ \emph{\width}_R(M) = \emph{\depth}_R(\emph{\Hom}_R(M,T)) $.}
			\item[(ii)]{If $M \in \mathcal{P}^{noet}(R) $, then $ \emph{\depth}_R(M) = \emph{\width}_R(\emph{\Hom}_R(M,T)) $.}
		\end{itemize}
\end{thm}
\begin{proof}
	In the rest of the proof, we can assume that $R$ is complete.

(i). Since $ \Ext^i_R(M,T) = 0 $ for all $ i > 0 $, by \cite[Theorem 10.62]{Ro}, there is a third quadrant spectral sequence \\
\centerline{$ \E^{p,q}_2 = \Ext^p_R(\Tor^q_R(k,M) , T) \underset{p}\Longrightarrow \Ext^n_R(k,\Hom_R(M,T)) $.}
Assume that $ \width_R(M) = t $. Then, we have $ \E^{p,q}_2 = 0 $ for all $ p > 0 $ whenever $ q < t $. Consequently, we get \\
\centerline{$ \Hom_R(\Tor^R_i(k,M),T) = \E^{0,i}_2 \cong  \Ext^i_R(k,\Hom_R(M,T)) $,}
for all $ 0 \leq i \leq t $. This shows that $  \Ext^i_R(k,\Hom_R(M,T)) = 0 $ for all $ i < t $ and that 
$  \Ext^t_R(k,\Hom_R(M,T)) \neq 0 $, whence $ \depth_R(\Hom_R(M,T)) = t $.

(ii). One has
  	\[\begin{array}{rl}
  	\depth_R(M) &= \depth_R\Big(\Hom_R(\Hom_R(M,T),T)\Big)\\
  	&=\width_R(\Hom_R(M,T)), \\
  	  	\end{array}\]
in which the first equality is from Lemma 2.5 and the other follows from the part (i) and Theorem 4.14.
\end{proof}
\section{Characterizations}
In this section, first, we define a new notion for artinian modules that is dual to the notion of the type for modules, namely cotype of modules. Next, we give some basic facts about this notion and use them to characterize dualizing modules.
\begin{defn}
\emph{ Let $M$ be an artinian $R$-module with $\width_R(M) = t $. The cotype of $M$, denoted by $ j_R(M) $ is defined as 
	$ j_R(M) = \vdim_k\big(\Tor^R_t(k,M)\big) $. }
\end{defn}   
\begin{rem}
\emph{Let $M$ be artinian, $ x \in \fm$ an $R$-regular and $M$-coregular element. Then, according to Lemma 4.8, we have the equalities \\ 
	\centerline{$ j_R(M) = j_R(\Hom_R(R/xR,M)) = j_{R/xR}(\Hom_R(R/xR,M))  $.}
	Also, the equality $ j_R(M) = j_{\widehat{R}}(M) $ holds. Moreover, the equality $ j_R(M) = r_R(M^{\vee}) $ holds, since by \cite[Theorem 3.2.13]{EJ1}, we have the isomorphism $ \Tor^R_i(k,M)^{\vee} \cong \Ext_R^i(k,M^{\vee}) $ for all $ i \geq 0 $.}
\end{rem}
In \cite{RB}, R.N.Roberts has introduced a dual notion of dimension
for Artinian modules. In \cite{I}, M.Inoue has introduced the notion of \textit{co-Cohen-Macaulayness} for artinian modules. Finally, in \cite{Y2}, S.Yassemi has defined a dual notion of dimension for modules, namely the \textit{Magnitude} of modules, which agrees with the definition of R.N.Roberts for artinian modules. Following \cite{Y1}, an $R$-module $M$ is \textit{cocyclic} if it is a submodule of $ E(R/\fm) $ for some $ \fm \in \max(R) $. Also, if  $M$ is an $R$-module, then we say that $\fp \in \Spec(R)$ is a \textit{coassociated prime ideal} of $M$, precisely when there exists a cocyclic homomorphic image $N$ of $ M$ for which $ \fp = \Ann_R(N) $. The set of coassociated prime ideal of $M$ is denoted by $ \cass_R(M) $. Following \cite{Y2}, the magnitude of $M$, denoted by
$ \mg_R(M) $, is defined to be the supremum of $ \dim_R(R/\fp) $ where $\fp$ runs over the set $\cass_R(M)$. An $R$-module $M$ is said to be  \textit{co-Cohen-Macaulay} (abbreviated to CCM), precisely when $ \mg_R(M) = \width_R(M) $.
\begin{defn}
\emph{Let $M$ be a non-zero artinian $R$-module. We say that $M$ is a maximal co-Cohen-Macaulay (abbreviated to MCCM), precisely when
$ \mg_R(M) = \width_R(M) = \dim(R) $.}
\end{defn}
For example, if $R$ is Cohen-Macaulay, then $ E(R/\fm) $ is a MCCM $R$-module. In fact, the codualizing $R$-module, if exists, is MCCM.
\begin{cor}
Let $M$ be a \emph{MCCM} $R$-module of finite projective dimension and let $T$ be codualizing. Then $M \cong T^n$ for some $ n \in \mathbb{N} $.	In particular, if $ j_R(M) = 1 $, then  $ M $ is codualizing.
\end{cor}
\begin{proof}
Just use Theorem 4.12(i).
\end{proof}
\begin{lem}
	Let $T$ be codualizing and $ d = \emph{dim}(R) $.
	Then $ \emph{\Ext}^i_R(T,k) = 0 =  \emph{\Tor}_i^R(T,k)$ for all $ i \neq d $ and that $ \emph{\Ext}^d_R(T,k) \cong k \cong  \emph{\Tor}_d^R(T,k)$
\end{lem}
\begin{proof}
We can assume that $R$ is complete.	By \cite[Corollary 4.5]{RT2} and \cite[Theorem 3.1]{K}, the minimal flat resolution of $T$ is of the form \\
	\centerline{$ 0 \longrightarrow \widehat{R_{\fm}} \longrightarrow \underset{\emph{\h}(\fp) = d-1} \prod T_{\fp} \longrightarrow \cdots  \longrightarrow \underset{\emph{\h}(\fp) = 0} \prod T_{\fp} \longrightarrow T \longrightarrow 0 $,}
	in which  $ T_{\fp} $ is the completion of a free
	$R_{\fp}$-module with respect to $ \fp R_{\fp} $-adic topology.
	For any prime ideal $\fp \neq \fm$,  we can take an element
	$ x \in \fm \smallsetminus \fp $. Now, multiplication of $ x $ induces an automorphism on $ E(R/ \fp) $ and hence on $ \prod T_{\fp} $, and that $ xk = 0 $. Consequently, multiplication of $ x $ on $ k \otimes_R \Big( \prod T_{\fp} \Big) $ is both an isomorphism and zero. Therefore, we have $ \Tor^R_i(T,k) = 0 $ for all $ i \neq d $. On the other hand, we have \\
	\centerline{$ \Tor^R_d(T,k) = \widehat{R_{\fm}} \otimes_R k \cong k $.}
	Finally, for all $ i \geq 0 $, we have\\
	\centerline{$ \Tor^R_i(T,k)^{\vee} \cong \Ext^i_R(T,k^{\vee}) \cong \Ext^i_R(T,k)$,}
	in which the first isomorphism is from \cite[Theorem 3.2.1]{EJ1}. Now, since $ k \cong k^{\vee} $, we get the desired isomorphisms.
\end{proof}
\begin{lem}
	Let $M$ be an artinian \emph{CCM} $R$-module with $ \emph{\mg}_R(M)=t $ and let $T$ be codualizing and $ d = \emph{dim}(R) $.
	Then $ \emph{\Ext}^i_R(T,M) = 0 $ for all $ i \neq d-t $ and that $ \emph{\Ext}^{d-t}_R(T,M) $ is Cohen-Macaulay of dimension $ t $.
\end{lem}
\begin{proof}
	Induct on $ t $. If $ t = 0 $, then $M$ is of finite length by \cite[Lemma 2.9]{Y2}. Now, by using a composition series for $M$ in conjunction with Lemma 5.5, we have $ \Ext^i_R(T,M) = 0$ for all $ i < d $. Also, by Proposition 4.1(ii), $ \Ext^i_R(T,M) = 0$ for all $ i > d $. Moreover, since $ \Ext^d_R(T,M) $ is of finite length, $ \Ext^d_R(T,M) $ is Cohen-Macaulay of dimension $ 0 $. Now, assume inductively that $ t > 0 $. Then, we can find an element $ x \in \fm $ which is $M$-coregular. The exact sequence \\
		\centerline{$ 0 \longrightarrow \Hom_R(R/xR,M) \longrightarrow M \overset{x}\longrightarrow M \longrightarrow 0 $,}
		induces a long exact sequence \\
\centerline{$ \cdots \overset{x}\rightarrow \Ext^{i-1}_R(T,M) \rightarrow \Ext^i_R(T,\Hom_R(R/xR,M)) \rightarrow \Ext^i_R(T,M) \overset{x}\rightarrow \Ext^i_R(T,M) \rightarrow \cdots $.}
Since, by \cite[Theorem 3.3]{Y2}, $ \Hom_R(R/xR,M) $ is CCM of dimension $ t-1 $, by induction hypothesis, we have $ \Ext^i_R(T,\Hom_R(R/xR,M)) = 0$ for all 
$ i \neq d-(t-1) $ and $ \Ext^{d-(t-1)}_R(T,\Hom_R(R/xR,M)) $  is Cohen-Macaulay of dimension $t-1$. Hence, by Nakayama's lemma, we have $ \Ext^i_R(T,M) = 0$ for all $ i \neq d - t $. Finally, the exact sequence \\
\centerline{$ 0 \rightarrow \Ext^{d-t}_R(T,M) \overset{x}\rightarrow \Ext^{d-t}_R(T,M) \rightarrow \Ext^{d-(t-1)}_R(T,\Hom_R(R/xR,M)) \rightarrow 0 $,}
shows that $ \Ext^{d-t}_R(T,M)/x\Ext^{d-t}_R(T,M) $ is  Cohen-Macaulay of dimension $t-1 $, whence $ \Ext^{d-t}_R(T,M) $ is  Cohen-Macaulay of dimension $ t $, as wanted.
\end{proof}
\begin{cor}
		Let $M$ be an Cohen-Macaulay $R$-module with $ \emph{\dim}_R(M)=t $ and let $T$ be codualizing and $ d = \emph{dim}(R) $.
		Then $ \emph{\Tor}_i^R(T,M) = 0 $ for all $ i \neq d-t $ and that $ \emph{\Tor}_{d-t}^R(T,M) $ is \emph{CCM} of magnitude $ t $.
\end{cor}
\begin{proof}
	Note that $ \depth_R(M) = \width_R(M^{\vee}) $, and by \cite[Lemma 2.2]{Y2} we have the equality $ \dim_R(M) = \mg_R(M^{\vee}) $. Hence $ M^{\vee} $ ic CCM $R$-module of magnitude $ t $. Thus, by Lemma 5.6, we have  $ \Ext^i_R(T,M^{\vee}) = 0$ for all $ i \neq d - t $ and that $ \Ext^{d-t}_R(T,M^{\vee}) $ is  Cohen-Macaulay of dimension $ t $. On the other hand, by \cite[Theorem 3.2.1]{EJ1}, we have the isomorphism 
	$ \Tor_i^R(T,M)^{\vee} \cong \Ext^i_R(T,M^{\vee}) $ for all $ i \geq 0 $. This completes the proof.
\end{proof}
\begin{cor}
	Let $M$ be \emph{MCCM} $R$-module and let $T$ be codualizing. Then we have the isomorphism $ T \otimes_R \emph{\Hom}_R(T,M) \cong M $.
\end{cor}
\begin{proof}
	We can assume that $R$ is complete. Since $ \Ext^i(T,T) = 0 $ for all $ i > 0 $, we can consider the following spectral sequence \\
	\centerline{$ \E^{p,q}_2 = \Tor_p^R(T , \Ext^q_R(T,M)) \underset{p}\Longrightarrow \Ext^n_R(R,M) $.}
	By Lemma 5.6, $  \Ext^q_R(T,M) = 0 $ for all $ q > 0 $ and that $ \Hom_R(T,M) $ is MCM. Hence, Corollary 5.7 yields that
	$  \Tor_p^R(T,\Hom_R(T,M)) = 0 $ for all $ p > 0 $. Thus, we get the isomorphism \\
	\centerline{$ T \otimes_R \Hom_R(T,M) \cong \Hom_R(R,M) \cong M $,}
	as wanted.
\end{proof}
\begin{thm}
Let $M$ be an artinian \emph{CCM} $R$-module with $ \emph{\mg}_R(M)=t $ and let $T$ be codualizing and $ d = \emph{dim}(R) $.
	\begin{itemize}
		\item[(i)]{$j_R(M) = \nu_R\big( \emph{\Ext}^{d-t}_R(T,M) \big) = \emph{\vdim}_k \Big(\emph{\Soc}_R\big(\emph{\Tor}_{d-t}^R(T,M^{\vee}) \big)\Big)$}
		\item[(ii)]{$\emph{\vdim}_k \big(\emph{\Soc}_R(M)\big) = r_R\big( \emph{\Ext}^{d-t}_R(T,M) \big) = j_R\big(\emph{\Tor}_{d-t}^R(T,M^{\vee}) \big)$.}
	\end{itemize}
\end{thm}
\begin{proof}
	In the rest of the proof, we can assume that $R$ is complete, and so $ T \cong \omega_R^{\vee} $.
	
	(i). One has
		\[\begin{array}{rl}
		j_R(M) &= r_R(M^{\vee})\\
		&= \nu_R\big(\Ext^{d-t}_R(M^{\vee},\omega_R)\big) \big(= \nu_R \big(\Ext^{d-t}_R(T,M)\big)\big) \\
		&= \vdim_k\Big(\Soc\big(\Ext^{d-t}_R(M^{\vee},\omega_R)^{\vee}\big)\Big) \\
		&= \vdim_k\Big(\Soc\big(\Tor_{d-t}^R(T,M^{\vee})\big)\Big), \\
		\end{array}\]
		in which the first equality is from Remark 5.2 and the second one is from \cite[Proposition 3.3.11(b)]{BH}.
		
		(ii). One has
			\[\begin{array}{rl}
		\vdim_k \big(\Soc_R(M)\big) &= \nu_R(M^{\vee})\\
		&= r_R\big(\Ext^{d-t}_R(M^{\vee},\omega_R)\big) \big(= r_R \big(\Ext^{d-t}_R(T,M)\big)\big)\\
		&= j_R\big(\Ext^{d-t}_R(M^{\vee},\omega_R)^{\vee}\big) \\
		&= j_R\big(\Tor_{d-t}^R(T,M^{\vee})\big), \\
		\end{array}\]
		in which the second equality is from \cite[Proposition 3.3.11(a)]{BH} and third one is from Remark 5.2.
\end{proof}
\begin{cor}
	Let $R$ be Cohen-Macaulay and let $M$ be an artinian $R$-module. The following are equivalent:
		\begin{itemize}
			\item[(i)]{$M$ is codualizing.}
			\item[(ii)]{$M$ is a \emph{MCCM} $R$-module with $ j_R(M) = 1 $ and $ \emph{\Ann}_R(M) = 0 $.}
		\end{itemize}
		\begin{proof}
			(i) $ \Longrightarrow $ (ii). Is clear by definition and Lemma 5.5.
			
			(ii) $ \Longrightarrow $ (i). We can assume that $R$ is complete. So that, by Remark 4.13, $R$ admits a codualizing module, say $T$. By Theorem 5.9(i), we have $ \nu_R \big(\Hom_R(T,M)\big) = 1 $. Hence $ \Hom_R(T,M) \cong R/\fa $ for some ideal $ \fa $ of $R$. Therefore, by Corollary 5.8, we have $ M \cong T \otimes_R \Hom_R(T,M) \cong T/\fa T $. It
			 follows that $ \fa M  = 0$ which implies that $ \fa = 0 $. Consequently, $ M \cong T $ is codualizing.
		\end{proof}
\end{cor}
\begin{prop}
Let $R$ be complete and $ d = \emph{\dim}(R) $.	Let $T$ be codualizing and let $D$ be dualizing. Then
	$$\emph{Ext}_R^i(T,D) \cong \left\lbrace
	\begin{array}{c l}
	D\ \ & \text{ \ \ \ \ \ $i=d$,}\\
	0\ \   & \text{  \ \ $\text{ \ \ $i \neq d$}$.}
	\end{array}
	\right.$$
	and 
		$$\emph{Tor}^R_i(T,D) \cong \left\lbrace
		\begin{array}{c l}
		E(R/\fm)\ \ & \text{ \ \ \ \ \ $i=0$,}\\
		0\ \   & \text{  \ \ $\text{ \ \ $i > 0$}$.}
		\end{array}
		\right.$$
\end{prop}
\begin{proof}
	The minimal injective resolution of $D$ is of the form \\
	\centerline{$ 0 \rightarrow D \rightarrow \underset{\h(\fp) = 0} \bigoplus E(R/ \fp) \rightarrow  \cdots \rightarrow \underset{\h(\fp) = d-1} \bigoplus E(R/ \fp) \rightarrow E(R/ \fm) \rightarrow 0 $.}
	Now, for the first isomorphism, just note that $ \Hom_R(T,E(R/\fp))  = 0$ for all $ \fp \neq \fm $ and that $ T^{\vee} \cong D $. Next, note that $ \Tor^R_i(T,D) = 0 $ for all $ i > 0 $. Finally, we have the isomorphisms $ T \otimes_R D \cong D^{\vee} \otimes_R D \cong E(R/\fm)$.
\end{proof}
\begin{thm}
Let $T$ ba a quasidualizing $R$-module and let $\emph{dim}(R)=d$. The following are equivalent:
		\begin{itemize}
			\item[(i)]{$T$ is codualizing.}
			\item[(ii)]{One has
				  $$\emph{Tor}^R_i(T,\emph{Ext}^j_R(T,k)) \cong \left\lbrace
				\begin{array}{c l}
				k\ \ & \text{ \ \ \ \ \ $i=j=d$,}\\
				0\ \   & \text{  \ \ $\text{ \ \ otherwise}$.}
				\end{array}
				\right.$$
			}
				\item[(iii)]{One has
					$$\emph{Ext}_R^i(T,\emph{Tor}^R_i(T,k)) \cong \left\lbrace
					\begin{array}{c l}
					k\ \ & \text{ \ \ \ \ \ $i=j=d$,}\\
					0\ \   & \text{  \ \ $\text{ \ \ otherwise}$.}
					\end{array}
					\right.$$
				}
		\end{itemize}
\end{thm}
\begin{proof}
	In the rest of the proof, we can assume that $R$ is complete.
	
	(i) $ \Longrightarrow $ (ii) and (i) $ \Longrightarrow $ (iii) are clear by Lemma 5.5.
	 	
	 	(ii) $ \Longleftrightarrow $ (iii).
	 	By \cite[Corollary 2.3 and Theorem 3.1]{KLW1}, $ \Tor^R_i(T,\Ext^j_R(T,k)) $ and $ \Ext^i_R(T,\Tor^R_j(T,k)) $ are finite dimensional $k$-vector spaces for all $ i \geq 0 $ and $ j \geq 0 $. This explains the first isomorphism in the following display 
	 		\[\begin{array}{rl}
	 		\Tor^R_i(T,\Ext^j_R(T,k)) &\cong \Tor^R_i\big(T,\Ext^j_R(T,k)\big)^{\vee}\\
	 		&\cong \Ext^i_R\big(T,\Ext^j_R(T,k)^{\vee}\big) \\
	 		&\cong \Ext^i_R\big(T,\Ext^j_R(k,T^{\vee})^{\vee}\big) \\
	 		&\cong  \Ext^i_R \big(T,\Tor^R_j(T^{\vee \vee},k)\big) \\
	 		&\cong  \Ext^i_R(T,\Tor^R_j(T,k)), \\
	 		\end{array}\]
	 and the other isomorphisms are from \cite[Theorem 3.2.1]{EJ1} and the fact that $ k \cong k^{\vee} $,
	 	for all $ i \geq 0 $ and $ j \geq 0 $.	
	
	(ii) $ \Longrightarrow $ (i).
	By \cite[Corollary 2.3]{KLW1}, we have $ \Ext_R^i(T,k) \cong k^{s_i} $ and by \cite[Theorem 3.1]{KLW1}, we have $ \Tor^R_i(T,k) \cong k^{t_i} $ for some $ s_i,t_i \in \mathbb{N} $.  Hence, by assumption, we have $ \Tor^R_i(T,k) = 0 $ for all $ i \neq d $ and 
	$s_d = 1 = t_d$. Thus, $ \Ext^i_R(k,T^{\vee}) \cong \Tor^R_i(T,k)^{\vee} = 0  $ for all $ i \neq d $ and 
	$  \Ext^d_R(k,T^{\vee}) \cong k^{\vee} \cong k$. Therefore $T^{\vee}$ is dualizing, whence $T$ is codualizing, as wanted.
\end{proof}
\begin{cor}
Let $R$ be a $d$-dimensional Cohen-Macaulay ring.
	\begin{itemize}
		\item[(i)]{One has
			$$\emph{Tor}^R_i\big(\emph{\H}^d_{\fm}(R),\emph{Ext}^j_R(\emph{\H}^d_{\fm}(R),k)\big) \cong \left\lbrace
			\begin{array}{c l}
			k\ \ & \text{ \ \ \ \ \ $i=j=d$,}\\
			0\ \   & \text{  \ \ $\text{ \ \ otherwise}$.}
			\end{array}
			\right.$$
		}
		\item[(ii)]{One has
			$$\emph{Ext}_R^i\big(\emph{\H}^d_{\fm}(R),\emph{Tor}^R_i(\emph{\H}^d_{\fm}(R),k)\big) \cong \left\lbrace
			\begin{array}{c l}
			k\ \ & \text{ \ \ \ \ \ $i=j=d$,}\\
			0\ \   & \text{  \ \ $\text{ \ \ otherwise}$.}
			\end{array}
			\right.$$
		}
	\end{itemize}
\end{cor}
\begin{proof}
Again, we can assume that $R$ is complete. Now, $R$ has a dualizing module $ \omega_R $ and, in addition, by local duality \cite[Theorem 11.2.8]{BS}, we have $ \H^d_{\fm}(R) \cong \omega_R^{\vee} $. Thus $ \H^d_{\fm}(R) $ is codualizing, and the result follows by Theorem 5.12.
\end{proof}
\begin{cor}
For a $d$-dimensional ring $R$, the following are equivalent:
 		\begin{itemize}
 			\item[(i)]{$R$ is Gorenstein.}
 			\item[(ii)]{One has
 				$$\emph{Tor}^R_i\big(E(R/\fm),\emph{Ext}^j_R(E(R/\fm),k)\big) \cong \left\lbrace
 				\begin{array}{c l}
 				k\ \ & \text{ \ \ \ \ \ $i=j=d$,}\\
 				0\ \   & \text{  \ \ $\text{ \ \ otherwise}$.}
 				\end{array}
 				\right.$$
 			}
 			\item[(iii)]{One has
 				$$\emph{Ext}_R^i\big(E(R/\fm),\emph{Tor}^R_i(E(R/\fm),k)\big) \cong \left\lbrace
 				\begin{array}{c l}
 				k\ \ & \text{ \ \ \ \ \ $i=j=d$,}\\
 				0\ \   & \text{  \ \ $\text{ \ \ otherwise}$.}
 				\end{array}
 				\right.$$
 			}
 		\end{itemize}
\end{cor}
\begin{proof}
	In Theorem 5.12, set $ T = E(R/\fm) $ and note that $ R $ is Gorenstein if and only if $ E(R/\fm) $ is codualizing by Remark 4.5.
\end{proof}
\bibliographystyle{amsplain}

\begin{thebibliography}{9}


\bibitem{BH}
~W. Bruns and ~J. Herzog , \emph{Cohen-Macaulay rings,} Cambridge University Press, Cambridge, 1993.


\bibitem{BS}
~M.P. Brodmann and ~R.Y. Sharp , \emph{Local cohomology: an algebraic introduction with geometric applica-
tions}, Cambridge University Press, Cambridge, 1998.



\bibitem{C}
~L.W. Christensen,  \emph{Gorenstien Dimensions,} Lecture Notes in Math., vol. 1747, Springer, Berlin, 2000.



\bibitem{EJ1}
~E. Enochs and ~O. Jenda, \emph{Relative Homological Algebra}, de Gruyter Expositions in
Mathematics 30, 2000.


\bibitem{EJ2}
~E. Enochs and ~O. Jenda, \emph{Tensor and torsion product of injective modules}, J. Pure Appl. Algebra \textbf{76} (1991), 143--149.


\bibitem{F}
~H.-B. Foxby, \emph{Gorenstein modules and related modules}, Math. Scand. \textbf{31} (1973), 267--284.


\bibitem{FW}
~A. Frankild and ~S. Sather-Wagstaff, \emph{The set of semidualizing complexes
is a nontrivial metric space,} J. Algebra \textbf{308} (2007), 124--143.


\bibitem{G}
~E. S. Golod, \emph{$G$-dimension and generalized perfect ideals}, Trudy Mat. Inst. Steklov. Algebraic geometry and its applications \textbf{165} (1984), 62--66.


\bibitem{H1}
~R. Hartshorne, \emph{Local cohomology}, A seminar given by A. Grothendieck, Harvard University,
Fall, vol. 1961, Springer-Verlag, Berlin, 1967.



\bibitem{HW}
~H. Holm, ~D. White, \emph{Foxby equivalence over associative rings,} J. Math. Kyoto Univ. \textbf{47} no.4,
(2007), 781--808.


\bibitem{I}
~M. Inoue, \emph{On Artinian modules,} Proc. Japan Acad. \textbf{62} (1986), 205--208.


\bibitem{K}
~B. Kubik, \emph{ Quasidualizing modules}, J. Commut. Algebra. \textbf{6} (2014), 209--229.


\bibitem{KLW1}
~B. Kubik, ~M. J. Leamer, and ~S. Sather-Wagstaff, \emph{ Homology of Artinian and
	Matlis reflexive modules} I, Journal of Pure and Applied Algebra \textbf{215} (2011), 2486--2503



\bibitem{KLW}
~B. Kubik, ~M. J. Leamer, and ~S. Sather-Wagstaff, \emph{ Homology of Artinian and
	mini-max modules} II, J. Algebra. \textbf{403}, no.1, (2014), 229--272.


\bibitem{GR}
~M. Raynaud and ~L. Gruson, \emph{Crit`eres de Platitude et de Projectivit´e,} Invent. Math.  \textbf{13} (1971), 1--89.


\bibitem{RB}
~R.N. Roberts, \emph{Krull dimension for Artinian modules over quasi local commutative
rings,} Quart. J . Math. Oxford. \textbf{26} (1975), 269--273.


\bibitem{Ro}
~J.J. Rotman, \emph{An introduction to homological algebra.}  Second ed., Springer, New York, 2009.


\bibitem{Ro2}
~J.J. Rotman, \emph{An introduction to homological algebra.}  Academic Press, 1979.



\bibitem{RT}
~M. Rahmani, ~A.-j.Taherizadeh, \emph{Characterizing local rings via perfect and coperfect modules,} Journal of Algebra and Its Applications, to appear.


\bibitem{RT2}
~M. Rahmani, ~A.-J. Taherizadeh, \textit{Dual of Bass numbers and dualizing modules}, Communication in Algebra, to appear.



\bibitem{S}
~R.Y. Sharp, \emph{Finitely generated modules of finite injective dimension over certain Cohen-Macaulay rings} Proc. London Math. Soc. \textbf{25} (1972) 303--328.


\bibitem{TW}
 ~R. Takahashi and ~D.White, \emph{Homological aspects of semidualizing modules}, Math. Scand. \textbf{106} (2010)
5--10.


\bibitem{V}
~W. V. Vasconcelos, \emph{Divisor theory in module categories,} North-Holland Math. Stud. 14, North-Holland
Publishing Co., Amsterdam (1974).

\bibitem{Y1}
~S. Yassemi, \emph{Coassociated primes,} Comm. Algebra, \textbf{23} (1995), 1473--1498.


\bibitem{Y2}
~S. Yassemi, \emph{Magnitude of modules,} Comm. Algebra, \textbf{23}  (1995), 3993--4008

\end{thebibliography}

\end{document}